\documentclass[oneside,english]{amsart}
\usepackage[T1]{fontenc}
\usepackage[latin9]{inputenc}
\usepackage{array}
\usepackage{multirow}
\usepackage{amsthm}
\usepackage{amstext}
\usepackage{amssymb}

\makeatletter

\providecommand{\tabularnewline}{\\}

\numberwithin{equation}{section}
\numberwithin{figure}{section}
  \theoremstyle{plain}
  \newtheorem*{thm*}{\protect\theoremname}
\theoremstyle{plain}
\newtheorem{thm}{\protect\theoremname}
  \theoremstyle{plain}
  \newtheorem{prop}[thm]{\protect\propositionname}
  \theoremstyle{plain}
  \newtheorem{lem}[thm]{\protect\lemmaname}
  \theoremstyle{plain}
  \newtheorem*{conjecture*}{\protect\conjecturename}

\usepackage{ulem}
\usepackage {url}
\usepackage [numbers]{natbib}

\makeatother

\usepackage{babel}
  \providecommand{\conjecturename}{Conjecture}
  \providecommand{\lemmaname}{Lemma}
  \providecommand{\propositionname}{Proposition}
  \providecommand{\theoremname}{Theorem}
\providecommand{\theoremname}{Theorem}

\begin{document}

\title{The Parity of Analytic Ranks among Quadratic Twists of Elliptic Curves
over Number Fields}

\author{Nava Balsam}

\date{November 07, 2013}
\begin{abstract}
The parity of the analytic rank of an elliptic curve is given by the
root number in the functional equation $L(E,s).$ Fixing an elliptic
curve over any number field and considering the family of its quadratic
twists, it is natural to ask what the average analytic rank in this
family is. A lower bound on this number is given by the average root
number. In this paper, we investigate the root number in such families
and derive an asymptotic formula for the proportion of curves in the
family that have even rank. Our results are then used to support a
conjecture about the average analytic rank in this family of elliptic
curves.
\end{abstract}
\maketitle

\section{Introduction}

Associated to an elliptic curve $E$ over a field $K$ is a family
of curves, its quadratic twists. Within this family, it is natural
to consider the the distribution of the most intriguing invariant
of an elliptic curve, its Mordell-Weil rank, which, via the Birch
and Swinnerton-Dyer conjecture is equal to its analytic rank. This
investigation began when Goldfeld \cite{Go79} conjectured that the
average analytic rank in such a family associated to an elliptic curve
over $\mathbb{Q},$ is $\frac{1}{2}$. The motivation for this conjecture
comes from considering the root number in the $L-$function of a twisted
elliptic curve and the folklore conjecture that elliptic curves of
rank $\ge2$ are rare. The result of this paper is a generalization
of the conjecture to elliptic curves over arbitrary number fields.
In many cases we find that the twists of $E/K$ do not have even and
odd analytic rank in equal proportion and therefore the conjectured
average value of the analytic rank is not $\frac{1}{2}$.

Recently, such an analysis was undertaken by Klagsbrun, Mazur and
Rubin \cite{KMR}, wherein they find the density of curves of even
Selmer rank among a family of quadratic twists. Using their results,
they formulate a generalization of Goldfeld's conjecture for elliptic
curves over number fields. Some of their techniques lend themselves
to the analytic aspect as well and have been a major inspiration for
this paper. We are able to verify their results in many cases, equating
the average algebraic rank (mod 2) with the average analytic rank
(mod 2). Thus, the result of this paper can be viewed as a parity
conjecture \textit{on average} in this family of elliptic curves.
In other words, the proportion of elliptic curves with even algebraic
rank is equal to the proportion of curves that have even analytic
rank, in a family of quadratic twists. 

Because of the convenience of dealing with automorphic representations
of $GL(2)$, in this paper we only consider elliptic curves which
are modular, a property which is defined as follows. Let $E$ be an
elliptic curve with conductor $\mathfrak{N}$ over a number field
$K$. Then $E$ is modular if the Hasse-Weil $L-$function of $E/K$
is equal to the $L-$function of a cuspidal Hilbert modular form of
parallel weight $2$, and level $\mathfrak{N}$ (See\cite{FS13}).
In particular, the Hasse-Weil $L-$function of $E$, denoted $L(E/K,s)$
or simply $L(E,s)$ is equal to the $L-$function of a cuspidal automorphic
representation $\pi_{E}=\otimes\pi_{E,v}$ of $GL(2,\mathbb{A}_{K})$
associated to a Hilbert modular form, and therefore $L(E,s)$ satisfies
a functional equation. It is known that all elliptic curves over $K=\mathbb{Q}$
are modular \cite{Mod01} and recently it has been shown that most
elliptic curves over totally real fields and real quadratic extensions
of totally real fields, are modular \cite{Hung13,FHS}.

Let $\Lambda_{E}(s)$ denote the ``completed'' $L$-function of $\pi_{E}$
, that is, including the factors at the archimedian places. Then the
functional equation takes the form $\Lambda_{E}(s)=w\Lambda(1-s)$
where $w\in\{\pm1\}$, called the root number will play a pivotal
role in this paper. The \textbf{analytic rank }of $E$, denoted $rk(E)$
is defined as 
\[
rk(E)=\text{order of vanishing of }\Lambda_{E}(s)\text{ at }s=\frac{1}{2}.
\]

For an elliptic curve $E/K$, we study the \textbf{quadratic twists}
of $E$: these are elliptic curves $E'$ which are isomorphic to $E$
over some quadratic extension $K'/K$. In order to make assertions
about density in a family of curves, it is necessary to order them
in some way. Suppose that $E'$ is isomorphic to $E$ over some quadratic
field $K'$. Then via standard class field theory, there is a unique
quadratic Hecke character $\chi$, associated to the extension $K'/K$.
Let $E^{\chi}$ be the curve which is isomorphic to $E$ over the
quadratic field $K'$ with associated Hecke character $\chi$ and
let $C(K)$ be the set of all quadratic Hecke characters of $\mathbb{A}_{K}^{\times}$.
For $\chi\in C(K)$, let $q_{1},q_{2},...,q_{n}$ be the places where
$\chi$ is ramified. Then we define the norm, $N\chi={\displaystyle \max_{i}}\{Nq_{i}\}$
which gives an ordering of twists of an elliptic curve. Our main result
is the following,
\begin{thm*}
Let $E$ be a modular elliptic curve over a number field $K$ such
that no local supercuspidal representations occur in the factorization
of $\pi_{E}$, then 
\[
{\displaystyle \lim_{X\rightarrow\infty}}\frac{\#\{\chi\in C(K)\mid N\chi\le X\text{ and rk}(E^{\chi})\text{ is even}\}}{\#\{\chi\in C(K)\mid N\chi\le\text{\ensuremath{X}\} }}=\frac{1+(-1)^{rk(E)}\kappa}{2}
\]
where $\kappa=\prod\kappa_{v}$ is a product over the places of $K$
given by 

\[
\kappa_{v}=\begin{cases}
0 & \text{if \ensuremath{K_{v}\backsimeq\mathbb{R}}}\\
1 & \ensuremath{\text{if }K_{v}\backsimeq\mathbb{C}}\\
2/|c_{v}|-1 & E\text{ has split mutiplicative reduction at \ensuremath{v}}\\
1-2/|c_{v}| & E\text{ has nonsplit multiplicative reduction at \ensuremath{v}}\\
1-2/|c_{v}| & E\text{ has multiplicative reduction in a quadratic extension at \ensuremath{v}}\\
1 & E\text{ has potentially multiplicative reduction (non-quadratic)}\\
1 & \text{otherwise}
\end{cases}
\]
 and $|c_{v}|$ is the number of degree 2 extensions of $K_{v}$ (if
$v\mid2$ then $|c_{v}|=4\cdot2^{[K_{v}:\mathbb{Q}_{v}]})$ otherwise
$|c_{v}|=4)$ . 
\end{thm*}
In particular the theorem holds unconditionally for semi-stable curves
over a real quadratic field\cite{Hung13}. Note that if the field
$K$ has a real embedding then the density of even analytic ranks
is exactly $\frac{1}{2}$.

\section{Notation}

Let $K$ be a number field and let $E$ be an elliptic curve over
$K$ with conductor $\mathfrak{N}.$ Let $\Lambda_{E}(s)$ be the
completed $L$-function of $E$, which, under the assumption of modularity,
is also the $L-$function of a cuspidal automorphic representation
of $GL(2,\mathbb{A}_{K})$. Let $\chi:K^{\times}\backslash\mathbb{A}_{K}^{\times}\rightarrow\pm1$
be a quadratic Hecke character with conductor $\mathfrak{f}$ and
let $\Lambda_{E}(s,\chi)=\Lambda_{E^{\chi}}(s)$ denote the $L$-function
of the quadratic twist $E^{\chi}$. Let $w$ $\in\{\pm1\}$ be the
``root number'' in the functional equation, i.e. such that $\Lambda_{E}(s)=w\Lambda_{E}(1-s)$
and let $n(\chi)$ be the ``change'' in the epsilon factor, i.e such
that $\Lambda_{E}(s,\chi)=n(\chi)\cdot\omega\Lambda_{E}(1-s,\chi)$
holds. Recall that $rk(E)$ denotes the analytic rank of $E$, as
defined in the previous section.

Let $C(K)$ be the group of global quadratic Hecke characters and
$C(K_{v})$ be the group of local quadratic characters at a place
$v$ of $K$. We define the norm of a Hecke character as $N\chi:={\displaystyle \max_{\begin{smallmatrix}\chi\text{ ramifies}\\
\text{at }q
\end{smallmatrix}}}\{Nq\}$. 

Let $C(K,X$) be the group of global Hecke characters such that $N\chi\le X$
and $\Gamma={\displaystyle \prod_{v\in\Sigma}C(K_{v})}$ where $\Sigma$
is a finite set of places including $\infty,$ and all $v\mid\mathfrak{N}$.

Let $E/K$ be a modular elliptic curve, then the L-function $\Lambda_{E}(s)$
is equal to the $L-$function $\Lambda(\pi_{E},s)$ attached to a
global automorphic representation $\pi_{E}$ of $GL(2,\mathbb{A}_{Q})$
arising from a Hilbert modular form. In particular, this implies that
$\Lambda(\pi_{E^{\chi}},s)=\Lambda(\chi\otimes\pi_{E},s)$.

Let $K_{v}$ be a local field. If $K_{v}$ is non-archimedian, let
$\varpi$ be a uniformizing element and let$\mathfrak{o}_{v}$ be
the ring of integers of $K_{v}$.

At a non-archimedian place $v,$ the admissible representations of
$GL(2,K_{v})$ are classified into 3 basic types of representation
\cite{JL}. The class of supercuspidal representations is excluded
from this paper because their local epsilon factors present difficulties
with respect to computation. The other two types, the principal series
representations, denoted $\pi(\mu_{1},\mu_{2})$ and the special representation
denoted $\sigma(\mu_{1},\mu_{2})$ are discussed in detail in \cite{JL}
which is also a major reference for the computational aspects in this
paper. Because elliptic curves are self-dual we have that $\mu_{1}=\mu_{2}^{-1}$
in both the supercuspidal and the special representations occurring
in the factorization of $\pi_{E}$.

\section{Change in Global Root Number}

The functional equation of an elliptic curve takes the form $\Lambda_{E}(s)=w\Lambda_{E}(1-s)$
where $w$ is either +1 or -1. Thus, the order of vanishing at $s=\frac{1}{2}$
is even if and only if $w=1$. Consider a quadratic twist of this
$L-$function $\Lambda_{E}(s,\chi)$, we expect the root number to
change by a prescribed amount 
\[
\Lambda_{E}(s,\chi)=n(\chi)w\Lambda_{E}(1-s,\chi)
\]

so that, 

\begin{equation}
rk(E^{\chi})\equiv rk(E)\mod2\iff n(\chi)=1\label{eq:iff_n_is_1}
\end{equation}

Next we compute $n(\chi)$ explicitly.
\begin{prop}
\label{prop:root-number-change}Let $\chi={\displaystyle \prod_{v}\chi_{v}}$
and let $\varpi$ denote a local uniformizer at $v$. Then 
\[
n(\chi)={\displaystyle \prod_{v}n_{v}(\chi_{v})}
\]
 where the $n_{v}$ are given explicitly by the following table:
\end{prop}
\begin{tabular}{|c|>{\centering}p{1.2in}|c|c|}
\hline 
i & Type of Representation &  & $n_{v}(\chi)$\tabularnewline
\hline 
1 & \multirow{2}{1.2in}{$\pi_{v}=\pi_{v}(\mu_{v},\mu_{v}^{-1})$ , $\mu_{v}$ is unramified,} & $\chi_{v}$ is unramified. & 1\tabularnewline
\cline{1-1} \cline{3-4} 
2 &  & $\chi_{v}$ is ramified. & $\chi(-1)$\tabularnewline
\hline 
3 & \multirow{3}{1.2in}{$\pi_{v}=\pi_{v}(\mu_{v},\mu_{v}^{-1})$ , $\mu_{v}$ is ramified} & $\chi_{v}$ is unramified. & 1\tabularnewline
\cline{1-1} \cline{3-4} 
4 &  & $\chi_{v}$ is ramified, $\mu_{v}\chi_{v}$ is unramified. & $\chi_{v}(-1)$\tabularnewline
\cline{1-1} \cline{3-4} 
5 &  & $\chi_{v}$ is ramified, $\mu_{v}\chi_{v}$ is ramified. & $\chi_{v}(-1)$\tabularnewline
\hline 
6 & \multirow{2}{1.2in}{$\pi_{v}=\sigma_{v}(\mu_{v},\mu_{v}^{-1})$ , $\mu_{v}$ is unramified, } & $\chi_{v}$ is unramified. & $\chi_{v}(\varpi)$\tabularnewline
\cline{1-1} \cline{3-4} 
7 &  & $\chi_{v}$ is ramified. & $-\chi_{v}(-1)\mu_{v}(\varpi)^{-1}$\tabularnewline
\hline 
8 & \multirow{3}{1.2in}{$\pi_{v}=\sigma_{v}(\mu_{v},\mu_{v}^{-1})$ , $\mu_{v}$ is ramified,} & $\chi_{v}$ is ramified, $\mu_{v}\chi_{v}$ is ramified. & $\chi_{v}(-1)$\tabularnewline
\cline{1-1} \cline{3-4} 
9 &  & $\chi_{v}$ is ramified, $\mu_{v}\chi_{v}$ is unramified. & $-\chi_{v}(-\varpi)\mu_{v}(\varpi)$\tabularnewline
\cline{1-1} \cline{3-4} 
10 &  & $\chi_{v}$ is unramified. & 1\tabularnewline
\hline 
\end{tabular}
\begin{proof}
Let $\mathfrak{N=}\prod\eta_{v}$ be the conductor of the elliptic
curve $E,$ and $\mathfrak{f}=\prod\mathfrak{f_{v}}$ , the conductor
of $\chi$ as in section 2, and let $\varpi_{v}$ be a local uniformizer
at $v$. It follows from the tensor-product theorem (\cite{GH} Chapter
10, \cite{JL}), that $n={\displaystyle \prod_{v}}n(\chi_{v})$, where
the local root number $n(\chi_{v})\in\{\pm1\}$ is such that 
\[
L_{v}(s,\chi\otimes\pi)=n(\chi_{v})\cdot w_{v}L_{v}(1-s,\chi\otimes\pi)
\]
 holds, where $w_{v}$ is the local root number of the (untwisted)
local $L$-function. In order to find $n(\chi_{v})$ we first compute
\[
a_{v}(s)=\frac{\epsilon(s,\pi_{v})}{\epsilon(s,\pi_{v}\otimes\chi_{v})}
\]
 and then for any $s\in\mathbb{R}_{>0}$ we have that
\[
n(\chi_{v})=\frac{a_{v}(s)}{|a_{v}(s)|}
\]
 The essence of this manipulation is that the root number is the ``sign''
of the epsilon factor. The proposition then follows from local epsilon
factor computations in each case as given in Jacquet-Langlands \cite{JL}.
As an illustration we include the proof in one of the above cases. 

The case presented here is when $\pi_{v}$ is an unramified representation
but $\chi_{v}\otimes\pi_{v}$ is not (line 2 in the chart). This is
equivalent to the statement that $v\nmid\mathfrak{N}$, $v\mid\mathfrak{f}$
i.e. that the representation $\pi_{v}$ is an unramified principal
series $\pi(\mu_{1},\mu_{2})$ and that $\chi_{v}$ is ramified with
conductor $\mathfrak{f}_{v}$. Choose an additive character of $K_{v}$,
say $\psi(x)=e^{2\pi i\Lambda(a/\mathfrak{df_{v}})}$, as defined
in Tate's thesis \cite{TT}. With the formulas as in \cite{JL}, we
compute: 
\[
\begin{array}{cclc}
\epsilon(s,\pi_{v}\otimes\chi_{v},\psi)^{-1} & = & \epsilon(s,\chi_{v}\mu_{1})^{-1}\epsilon(s,\chi_{v}\mu_{2})^{-1}\\
 & = & {\displaystyle \prod_{i=1,2}\frac{1}{\mu_{i}\chi_{v}(\mathfrak{df})}N(\mathfrak{df})^{s-\frac{1}{2}}N\mathfrak{f}^{-\frac{1}{2}}\sum\mu_{i}\chi_{v}(a)\psi(a)}\\
 & = & \frac{1}{\mu_{1}\mu_{2}(\mathfrak{df_{v})}}\epsilon(s,\chi_{v},\psi)^{-2}\\
 & = & \epsilon(s,\chi_{v},\psi)^{-2}
\end{array}
\]
the last formula following from the fact that $\mu_{1}\mu_{2}=1$
for all elliptic curves. We also have that 
\[
\epsilon(s,\pi_{v})=N(\mathfrak{d)}^{2s-1}
\]
 whenever $\pi_{v}$ is unramified. So that
\[
\begin{array}{cl}
a_{v} & =\frac{\epsilon(s,\chi_{v},\psi)^{-2}}{N(\mathfrak{d)}^{2s-1}}\\
 & =\frac{N\mathfrak{f_{v}}^{2s-2}}{\chi_{v}(\mathfrak{df}_{v})^{2}}\left({\displaystyle \sum_{a\mod\mathfrak{f}_{v}}}\chi(a)\psi(a)\right)^{2}\\
 & =N\mathfrak{f}_{v}^{2s-2}\tau(\chi_{v})^{2}\\
 & =N\mathfrak{f}_{v}^{2s-1}\chi_{v}(-1)
\end{array}
\]
Where $\tau(\chi_{v})$ is the guass sum of $\chi_{v}$ with respect
to the additive character $\psi$. The fact that $\tau(\chi_{v})^{2}=N\mathfrak{f}_{v}\chi(-1)$
can be shown using the epsilon factors for GL(1) (e.g. see \cite{BH},
section 23.) We have that $n_{v}(\chi_{v})=\chi(-1)$.

Let $S_{i}$ be the set of places that have the properties of the
$i$th row in the Table. For example, $S_{1}$ is the set of places
where $\pi_{v}$ is unramified and $\chi_{v}$ is unramified.\end{proof}
\begin{prop}
\label{prop:parity-iff}We have

\[
rk(E^{\chi})\equiv rk(E)\mod2\iff{\displaystyle \prod_{\underset{v\text{ real}}{v|\infty}}}\chi_{v}(-1){\displaystyle \prod_{S_{6}}}\chi_{v}(\varpi_{v}){\displaystyle \prod_{S_{7}}}\frac{-1}{\mu_{v}(\varpi)}{\displaystyle \prod_{S_{9}}}-\chi_{v}\mu_{v}(\varpi)=1
\]

where $S$ is the set of places where $\pi_{v}$ is a special representation.\end{prop}
\begin{proof}
By equation \ref{eq:iff_n_is_1} and Proposition \ref{prop:root-number-change}
we have that 
\[
\begin{matrix}rk(E^{\chi})\equiv rk(E)\mod2\\
\iff\\
{\displaystyle \prod_{S_{1},S_{3},S_{10}}}1{\displaystyle \prod_{S_{2},S_{4},S_{5},S_{8}}}\chi_{v}(-1){\displaystyle \prod_{S_{6}}}\chi_{v}(\varpi){\displaystyle \prod_{S_{7}}}\frac{-\chi_{v}(-1)}{\mu_{v}(\varpi)}{\displaystyle \prod_{S_{9}}}-\chi_{v}(-\varpi)\mu_{v}(\varpi)=1
\end{matrix}
\]

Now, since $\chi$ is a Hecke character, $\chi(-1,-1,...-1...)=\prod\chi_{v}(-1)=1$.
Therefore, we may multiply the right hand side by $\prod\chi_{v}(-1)$.
Now since $\chi_{v}$ is unramified in the sets $S_{1},S_{3},S_{10}$
and $S_{6},$ the result follows. 
\end{proof}
It will be useful to simplify the expression in Proposition \ref{prop:parity-iff}.
Let us denote $\Sigma_{1}=S_{6}\cup S_{7}$ and $\Sigma_{2}=S_{8}\cup S_{9}\cup S_{10}$
and let $m_{v}(\chi)=\chi_{v}(-1)a_{v}(\chi).$ More explicitly, 
\[
\begin{matrix}m_{v}(\chi)=\begin{cases}
\chi_{v}(\varpi) & \text{ if \ensuremath{\chi_{v}}is unramified}\\
-\mu_{v}(\varpi)^{-1} & \text{if }\ensuremath{\chi_{v}}\text{ is ramified}
\end{cases} & \text{for }v\in\Sigma_{1}\end{matrix}
\]

\[
\begin{matrix}m_{v}(\chi)=\begin{cases}
1 & \text{ if \ensuremath{\chi_{v}}is unramified}\\
1 & \text{if }\ensuremath{\chi_{v}}\text{ and \ensuremath{\chi_{v}\mu_{v}}are ramified}\\
-\chi_{v}\mu_{v}(\varpi) & \text{if }\ensuremath{\chi_{v}}\text{ is ramified and \ensuremath{\chi_{v}\mu_{v}}}\text{ is unramified}
\end{cases} & \text{for }v\in\Sigma_{2}.\end{matrix}
\]

We will rewrite Proposition \ref{prop:parity-iff} as 
\begin{equation}
rk(E^{\chi})\equiv rk(E)\mod2\iff{\displaystyle \prod_{\underset{v\text{ real}}{v|\infty}}}\chi_{v}(-1){\displaystyle \prod_{\Sigma_{1}\cup\Sigma_{2}}}m_{v}(\chi)\label{eq:parityeq_with_m}
\end{equation}

note that $\Sigma_{1}$ are places where $\pi_{v}$ is a special representation
with unramified character and $\Sigma_{2}$ are places where $\pi_{v}$
is a special representation with a ramified character. Thus, only
the places where special representations occur (or real places) change
the root number of a twisted curve.

\section{The Density of Even Analytic Ranks}

According to equation \ref{eq:parityeq_with_m}, the parity of the
analytic rank doesn't change upon twisting by $\chi$ if and only
a certain product of $-1's$ and $+1's$ occurring on the right-hand
side of \ref{eq:parityeq_with_m} is equal to $1$. The next step
involves computation the proportion of cases where the product is
$+1.$ This amounts to an exercise in counting quadratic characters.

Let $c_{v}$ be the set of local quadratic characters (if $v\nmid2$
then $|c_{v}|=4,$ and if $v\mid2$ then $|c_{v}|=4\cdot2^{[K_{v}:\mathbb{Q}_{v}]})$,
and let $\Sigma=\ensuremath{\Sigma_{1}\cup\Sigma_{2}}\cup\Sigma_{\mathbb{R}}$,
where $\Sigma_{\mathbb{R}}$ is the set of real places of $K$. For
each element of $\Sigma$, define 
\[
\kappa_{v}=\begin{cases}
\frac{1}{|c_{v}|}{\displaystyle \sum_{\chi\in c_{v}}}m_{v}(\chi), & \text{ for \ensuremath{v}in \ensuremath{\Sigma_{1}\cup\Sigma_{2}}}\\
\frac{1}{|c_{v}|}{\displaystyle \sum_{\chi\in c_{v}}}\chi_{v}(-1), & \text{ for real primes}
\end{cases}
\]

\begin{lem}
\label{lem:counting-trick} Let $\Gamma={\displaystyle \prod_{v\in\Sigma}C(K_{v})}$
. Then, 
\[
\frac{|\{\chi\in\Gamma\mid n(\chi)=1\}|}{|\Gamma|}=\frac{1+{\displaystyle \prod_{v\in\Sigma}}\kappa_{v}}{2}
\]
\end{lem}
\begin{proof}
(This is lemma 7.5 in \cite{KMR}) Let $N=|\{\chi\in\Gamma\mid n(\chi)=1\}|$.
Then$\Gamma$ can be written as ${\displaystyle \prod_{v\in\Sigma}c_{v}},$
and we have that 
\[
N-(|\Gamma|-N)={\displaystyle \sum_{\gamma\in\Gamma}\prod_{v\in\Sigma}n_{v}(\gamma_{v})=\prod_{v\in\Sigma}\sum_{\gamma_{v}\in c_{v}}n_{v}(\gamma_{v})}.
\]
 Now when we divide both sides by $\Gamma={\displaystyle \prod_{v\in\Sigma}}c_{v}$
we get that $2N/|\Gamma|-1={\displaystyle \prod_{v\in\Sigma}\kappa_{v}}$
and the lemma follows. 
\end{proof}
Next, we compute the values of $\kappa_{v}$.
\begin{lem}
\label{lem:The-local-factors-of-kappa}The local factors $\kappa$
have the following values\end{lem}
\begin{itemize}
\item $\kappa_{v}=0$ when $K_{v}\backsimeq\mathbb{R}$
\item $\kappa_{v}=\begin{cases}
2/|c_{v}|-1 & E\text{ has split mutiplicative reduction at \ensuremath{v}}\\
1-2/|c_{v}| & E\text{ has nonsplit multiplicative reduction at \ensuremath{v}}
\end{cases}$ for $v$ $\in\Sigma_{1}$
\item $\kappa_{v}=$$\begin{cases}
1-2/|c_{v}| & E\text{ has mult. red. in a quad. extension at \ensuremath{v}}\\
1 & \text{otherwise}
\end{cases}$ for $v$ $\in\Sigma_{2}$
\item $\kappa_{v}=$1 if $E$ has good reduction, potentially good reduction,
or $K_{v}\backsimeq\mathbb{C}$
\end{itemize}
Note that $\Sigma_{1}$ are the places where $E$ has multiplicative
reduction and $\Sigma_{2}$ are the places where $E$ has potentially
multiplicative reduction. Thus, $\Sigma_{1}\cup\Sigma_{2}\cup$\{real
places\} are the only places that affect the value of $\kappa$.
\begin{proof}
For a real place, there are exactly two characters: $\chi_{triv}$
and $\chi_{sign}$ and $\kappa_{\infty}=\frac{1}{2}(\chi_{triv}(-1)+\chi_{sign}(-1))=0.$ 

For a non-archimedian place in $\Sigma_{1}$, there is one trivial
character $\chi_{triv}$ and an unramified character $\chi_{u.r.}$.
For the first two, $m_{v}(\chi_{v})=\chi_{v}(\varpi)$, so that $m_{v}(\chi_{triv})=1$
and $m_{v}(\chi_{u.r.})=-1$. For the ramified characters $\chi_{ram},$
$m_{v}(\chi_{ram})=-\mu_{v}(\varpi_{v})$. In total we have
\[
\kappa_{v}=\frac{1}{|c_{v}|}\left(1-1+\left(|c_{v}|-2\right)\mu_{v}(\varpi)\right)
\]
It follows from a comparison of the $L-$function at multiplicative
reduction and at a special representation with unramified character
that $\mu_{v}(\varpi)=1$ at split multiplicative reduction and $\mu_{v}(\varpi)=-1$
at non-split multiplicative reduction, which gives the result for
$v\in\Sigma_{1}$.

For $v\in\Sigma_{2}$, we are in the situation where the local representation
is of the form $\sigma_{v}(\mu_{v},\mu_{v}^{-1})$ where $\mu_{v}$
is ramified. If $\chi_{v}$ is unramified, then $m_{v}(\chi_{v})=1$.
If $\chi_{v}$ is ramified and $\chi_{v}\mu_{v}$ is also ramified,
then $m_{v}(\chi_{v})=1.$ If $\chi_{v}\mu_{v}$ is unramified, then
$\mu_{v}\mid_{\mathfrak{o}_{k}^{\times}}$ must be a quadratic character
and then there are exactly two characters $\chi_{v,1}$ and $\chi_{v,2}$
such that $\chi_{v}\mu_{v}$ is unramified and furthermore, $\chi_{v,1}(\varpi)=1,$
and $\chi_{v,2}(\varpi)=-1$ so that $m_{v}(\chi_{v,1})=-\mu_{v}(\varpi)$
and $m_{v}(\chi_{v,2})=\mu_{v}(\varpi)$. In total, if $\mu_{v}\mid_{\mathfrak{o}_{k}^{\times}}$
is not quadratic, then $m_{v}(\chi_{v})=1$ for all local characters
$\chi_{v}$ and if $\mu_{v}\mid_{\mathfrak{o}_{k}^{\times}}$ is quadratic
then 
\[
\kappa_{v}=\frac{1}{|c_{v}|}\left(1\cdot\left(2-|c_{v}|\right)+\mu_{v}(\varpi)-\mu_{v}(\varpi)\right)=1-2/|c_{v}|
\]
 Now, if $\mu_{v}\mid_{\mathfrak{o}_{k}^{\times}}$ is quadratic,
then the elliptic curve has potential multiplicative reduction at
$v$ (in particular, multiplicative reduction in a quadratic extension
of $K_{v}).$ 
\end{proof}
The following lemma deals with the problem that not all collections
of local characters give rise to global characters.
\begin{lem}
\label{lem:surject-onto-local-prod}Recall that $C(K)$ is the group
of global quadratic characters, and $\Gamma={\displaystyle \prod_{v\in\Sigma}}C(K_{v})$.
The natural homomorphism $\alpha:C(K)\rightarrow\mbox{\ensuremath{\Gamma}}$
is surjective.\end{lem}
\begin{proof}
Let $\gamma=\prod\gamma_{v}\in{\displaystyle \prod_{v\in\Sigma}}C(K_{v})$
and let $s$ be a place of $K$, not in $\Sigma.$ Then if we set
$\gamma_{s}(q)=\gamma^{-1}(q)$, for all $q\in K^{\times}\subset K_{v}^{\times}$,
this defines a character on a dense subset of $K_{v}$, hence on all
of $K_{v}$ by continuity. Then $\gamma'=(\prod\gamma_{v})\cdot\gamma_{s}$
is trivial on $K^{\times}$ and $\alpha(\gamma')=\gamma.$
\end{proof}
The lemma allows us to convert statements about the density of even
analytic ranks among arbitrary products of local characters to statements
about the density of even analytic ranks among Hecke characters. 

To conclude, I will restate and prove the main theorem
\begin{thm*}
For all $X$ large enough, 
\[
\frac{\#\{\chi\in C(K,X)\text{ such that rk}(E^{\chi})\text{ is even}\}}{|C(K,X)|}=\frac{1+(-1)^{rk(E)}\kappa}{2}
\]
where $\kappa={\displaystyle \prod_{v}\kappa_{v}}$ are defined and
computed above.\end{thm*}
\begin{proof}
For $X$ large enough, the set of characters with norm less than $X$
surjects onto $\Gamma$ by lemma \ref{lem:surject-onto-local-prod}.
Now by Proposition \ref{prop:parity-iff} the analytic rank only depends
on a local product, and since the map $C(K)\rightarrow\Gamma$ is
a homomorphism, all its fibers have the same size. Therefore,
\[
\frac{\#\{\chi\in C(K,X)\text{ such that rk}(E^{\chi})=\text{rk}(E)\text{ }\}}{|C(K,X)|}=\frac{\#\{\chi\in\Gamma\text{ such that }n(\chi)=1\text{ }\}}{|\Gamma|}
\]

And now the theorem follows from lemma \ref{lem:counting-trick}
\end{proof}
Our theorem is about the expected parity of twists of $E$. The heuristic
that elliptic curves have rank as low as possible implies that given
the parity constraints, the curves in this family will have analytic
rank 0 or 1 with far greater frequency than the higher order ranks.
Thus together with the Birch and Swinnerton-Dyer conjecture, our results
support the conjecture in \cite{KMR}
\begin{conjecture*}
(Klagsbrun-Mazur-Rubin) 

\[
\lim_{X\rightarrow\infty}\frac{{\displaystyle \sum_{\chi\in C(K,X)}}rk(E^{\chi})}{|C(K,X)|}=\frac{1+(-1)^{rk(E)}\kappa}{2}
\]

\end{conjecture*}
Where the conjecture is adapted to \emph{analytic }ranks and the local
factors of $\kappa$ are explicitly computed in the cases of lemma
\ref{lem:The-local-factors-of-kappa}.

\section*{Acknowledgements}

I would like to thank Dorian Goldfeld for proposing this project and
for the numerous discussions and the copious amounts of advice and
guidance that led to its conclusion. I would also like to thank Karl
Rubin for helpful correspondence during the early drafts of this paper.

\bibliographystyle{plain}
\nocite{*}
\bibliography{ThePaperBib}

\end{document}